\documentclass{amsart}
\usepackage{amssymb,amsmath,amscd,amsthm}
\usepackage{amsfonts,epsfig,latexsym,graphicx,amssymb}

\date{\today}

\newcommand{\Z}{{\mathbb Z}}
\newcommand{\R}{{\mathbb R}}

\newcommand{\T}{{\mathbb T}}

\newtheorem{Definition}{Definition}
\newtheorem{theorem}{Theorem}[section]
\newtheorem{lemma}[theorem]{Lemma}

\newtheorem{coro}[theorem]{Corollary}

\theoremstyle{definition}
\newtheorem{remark}[theorem]{Remark}

\def\be{\begin{equation}}
\def\ee{\end{equation}}

\begin{document}

\title[Sums of Cantor sets and the Square Fibonacci Hamiltonian]{Sums of regular Cantor sets of large dimension and the Square Fibonacci Hamiltonian}

\author[D.\ Damanik]{David Damanik}

\address{Department of Mathematics, Rice University, Houston, TX~77005, USA}

\email{damanik@rice.edu}

\thanks{D.\ D.\ was supported in part by NSF grants DMS--1067988 and DMS--1361625.}

\author[A.\ Gorodetski]{Anton Gorodetski}

\address{Department of Mathematics, University of California, Irvine, CA~92697, USA}

\email{asgor@math.uci.edu}

\thanks{A.\ G.\ was supported in part by NSF grant DMS--1301515.} 

\begin{abstract}
We show that under natural technical conditions, the sum of a $C^2$ dynamically defined Cantor set with a compact set in most cases (for almost all parameters) has positive Lebesgue measure, provided that the sum of the Hausdorff dimensions of these sets exceeds one. As an application, we show that for many parameters, the Square Fibonacci Hamiltonian has spectrum of positive Lebesgue measure, while at the same time the density of states measure is purely singular.
\end{abstract}

\maketitle

\section{Introduction}\label{s.intro}

\subsection{Sums of dynamically defined Cantor sets}

The study of  the structure and the properties of sums of Cantor sets is motivated by applications in dynamical systems \cite{n1, n2, n3, PaTa}, number theory \cite{CF, Mal, Moreira}, harmonic analysis \cite{BM, BKMP}, and spectral theory \cite{EL06, EL07, EL08, Y}. In many cases {\it dynamically defined} Cantor sets are of special interest.

\begin{Definition}
A dynamically defined {\rm (}or regular{\rm )} Cantor set of class $C^r$ is a Cantor subset $C\subset\mathbb{R}$ of the real line such that there are disjoint compact intervals  $I_1, \ldots, I_l \subset \mathbb{R}$ and an expanding $C^k$ function $\Phi : I_1 \cup \cdots \cup I_l \to I$ from the disjoint union $I_1 \cup \cdots \cup I_l$ to its convex hull $I$ with
$$
C = \bigcap_{n\in\mathbb{N}}\Phi^{-n}(I).
$$
\end{Definition}

In the case when the restriction of the map $\Phi$ to each of the intervals $I_j, j=1, \ldots, l,$ is  affine, the corresponding Cantor set is also called {\it affine}. If all these affine maps have the same expansion rate (i.e., $|\Phi'(x)|=const$ for all $x\in I_1\cup\dots\cup I_l$), the Cantor set is called {\it homogeneous}.   A specific example of a homogeneous Cantor set, a {\it middle-$\alpha$} Cantor set $C_a$, is defined by $\Phi:[0, a]\cup [1-a, 1]\to [0,1]$, where $\Phi(x)=\frac{x}{a}$ for $x\in [0,a]$, and $\Phi(x)=\frac{x}{a}-\frac{1}{a}+1$ for $x\in [1-a,1]$. For example, $C_{1/3}$ is the standard middle-third Cantor set.

Considering the sum $C + C'$ of two Cantor sets $C,C'$, defined by
$$
C + C' = \{ c + c' : c \in C, \; c' \in C' \},
$$
it is not hard to show (see, e.g., Chapter 4 in \cite{PaTa}) that if the Cantor sets $C$ and $C'$ are dynamically defined, one has $\dim_\mathrm{H} (C + C') \le \min(\dim_\mathrm{H} C + \dim_\mathrm{H} C', 1)$. Hence in the case $\dim_\mathrm{H} C + \dim_\mathrm{H} C' < 1$, the sum $C + C'$ must be a Cantor set, and an interesting question here is whether the identity $\dim_\mathrm{H} (C + C') = \dim_\mathrm{H} C + \dim_\mathrm{H} C'$ holds. This question was addressed for homogeneous Cantor sets in \cite{PeSo} (see also \cite{NPS}), and some explicit criteria were provided in \cite{HS}.

In the case when $\dim_\mathrm{H} C + \dim_\mathrm{H} C' > 1$, a major result was obtained by Moreira and Yoccoz in \cite{MY}. They showed that for a {\it generic} pair of Cantor sets $(C, C')$ in this regime, the sum $C + C'$ contains an interval. The genericity assumptions there are quite non-explicit, and cannot be verified in a specific case. This does not allow one to apply this result when a specific pair or a specific family of Cantor sets is given (which is often the case in applications), which therefore motivates further investigations in this direction. For example, while \cite{MY} solves one part of the Palis conjecture on sums of Cantor sets (``generically the sum of two dynamically defined Cantor sets either has zero measure or contains an interval''), the second part of the conjecture (``generically the sum of two affine Cantor sets either has zero measure or contains an interval'') is still open.

An important characteristic of a Cantor set related to questions about intersections and sum sets is the {\it thickness}, usually denoted by $\tau(C)$. This notion was introduced by Newhouse in \cite{n1}; for a detailed discussion, see \cite{PaTa}. The famous {\it Newhouse Gap Lemma} asserts that if $\tau(C) \cdot \tau(C') > 1$, then $C + C'$ contains an interval. This allowed for essential progress in dynamics \cite{n2, n3, Du}, and found an application in number theory \cite{A}. Nevertheless, in some cases $\tau(C) \cdot \tau(C') < 1$, while $\dim_\mathrm{H} C + \dim_\mathrm{H} C' > 1$, and other arguments are needed.

In \cite{So97} Solomyak studied the sums $C_a + C_b$ of middle-$\alpha$ type Cantor sets. He showed that in the regime when $\dim_\mathrm{H} C + \dim_\mathrm{H} C' > 1$, for almost every pair of parameters $(a,b)$, one has $\mathrm{Leb}(C_a + C_b) > 0$. Similar results for sums of homogeneous Cantor sets (parameterized by the expansion rate) with a fixed compact set were obtained in \cite{PeSo}.

\medskip

In this paper we are able to work in far greater generality, and our first main result reads as follows:

\begin{theorem}\label{t.sums}
Let $\{C_\lambda\}$ be a family of dynamically defined Cantor sets of class $C^{2}$ {\rm (}i.e., $C_\lambda = C(\Phi_\lambda)$, where $\Phi_\lambda$ is an expansion of class $C^2$ both in $x\in \mathbb{R}$ and in $\lambda \in J=(\lambda_0, \lambda_1)${\rm )} such that $\frac{d}{d\lambda}\dim_\mathrm{H} C_\lambda \ne 0$ for $\lambda \in J$. Let $K\subset \mathbb{R}$ be a compact set such that
$$
\dim_\mathrm{H} C_\lambda + \dim_\mathrm{H} K > 1\ \ \ \text{for all}\ \ \ \lambda\in J.
$$
Then $\mathrm{Leb}(C_\lambda+K) > 0$ for a.e. $\lambda \in J$.
\end{theorem}

\begin{remark}
It would be interesting to relax the assumptions in Theorem \ref{t.sums} and to show that the same statement holds for $C^{1+\alpha}$ Cantor sets. We conjecture that this is indeed the case (possibly under some extra conditions on the dependence of $\Phi$ and $\frac{\partial}{\partial x}\Phi$ on $\lambda$). 
\end{remark}

In the case when the dynamically defined Cantor sets $\{C_\lambda\}$ are affine (or non-linear, but $C^2$-close to affine), a statement analogous to Theorem~\ref{t.sums} was obtained in \cite{GN}. The case of a sum of homogeneous (affine with the same contraction rate for each of the generators) Cantor sets with a dynamically defined Cantor set was considered in Theorem 1.4 in \cite{Shm}; in this case the set of exceptional parameters has zero Hausdorff dimension.

In order to put these results in perspective, we summarize them in the following table. (Below we always set $d_1 = \dim_\mathrm{H} C$ (or $d_1 = \dim_\mathrm{H} C_\lambda$) and $d_2 = \dim_\mathrm{H} K$. In the case $d_1 + d_2 < 1$ we ask whether it is true that $\dim_\mathrm{H} (C+K) = d_1 + d_2$, and in the case $d_1 + d_2 > 1$ we ask whether it is true that $C + K$ contains an interval, and if this is unknown, whether $\mathrm{Leb} (C + K) > 0$.)

\begin{center}

\section*{\bf What is known about sums of dynamically defined Cantor sets:}

\

{\small
\begin{tabular}{|l|c|c|c|}
            \hline
             \text{\ } & \text{\ }  & \text{\ }    \\
            \ \ \ \  \ \    & \ \  $d_1+d_2<1$   \  \  & \ \   $d_1+d_2>1$   \ \    \\
               \text{  } & \text{ } & \text{ }   \\
            \hline
             \text{ }  & \text{ }    & \text{ } \\
           \ \begin{minipage}{4cm} For a $C^r$ generic pair of Cantor sets $(C, K)$ \end{minipage}  & Yes, follows from \cite{HS}.  &   $C+K$ contains an interval \cite{MY}  \\
              \text{ }  & \text{ }    & \text{ } \\
              \hline
             \text{ } & \text{ }    & \text{ } \\
           \ \begin{minipage}{4cm} Given a family $\{C_\lambda\}_{\lambda \in J}$ such that $\frac{d}{d\lambda}(\dim_\mathrm{H} C_\lambda) \ne 0$, and a compact $K \subset \mathbb{R}$ \end{minipage} & \begin{minipage}{4cm} Yes, for a.e. parameter, follows from \cite{HS}. \end{minipage} & \begin{minipage}{4cm} $\mathrm{Leb}(C_\lambda + K) > 0$ for a.e.\ $\lambda$, this paper \end{minipage} \\
              \text{ } & \text{ }    & \text{ } \\
                        \hline
                          \text{ } & \text{ }    & \text{ } \\
           \ \begin{minipage}{4cm} For a generic pair of {\bf affine} Cantor sets $(C, K)$ \end{minipage}  & Yes, follows from \cite{HS}.  & \begin{minipage}{4cm} $\mathrm{Leb}(C+K) > 0$ \cite{GN}; whether $C+K$ contains an interval is an open part of the Palis conjecture \end{minipage}  \\
              \text{ } & \text{ }    & \text{ } \\
                        \hline
                          \text{ } & \text{ }    & \text{ } \\
           \ \begin{minipage}{4cm} Given a family $\{C_\lambda\}_{\lambda\in J}$ of {\bf homogeneous} self-similar Cantor sets and a compact $K\subset \mathbb{R}$  \end{minipage} & Yes, for a.e. parameter,  \cite{PeSo}  &  \begin{minipage}{4cm} $\mathrm{Leb}(C_\lambda + K) > 0$ for a.e. $\lambda$ \cite{PeSo} (see also \cite{Shm}); whether $C_\lambda+K$ contains an interval for a.e. parameter is unknown \end{minipage}    \\
              \text{ } & \text{ }    & \text{ } \\
                        \hline
                          \text{ } & \text{ }    & \text{ } \\
           \ \begin{minipage}{4cm} Given a family $\{C_a\}_{a\in (0, 1/2)}$ of {\bf middle-$\alpha$} Cantor sets, take $K=C_b$ for some $b\in (0, 1/2)$ \end{minipage} & \begin{minipage}{4cm} Yes, with countable number of exceptions, \cite{PeShm} \end{minipage}  & \begin{minipage}{4cm} $\mathrm{Leb}(C_a+K)>0$, for a.e. parameter $a$, \cite{So97}; whether $C_a+K$ contains an interval for a.e. parameter is unknown. \end{minipage}     \\
              \text{ } & \text{ }    & \text{ } \\
                        \hline
                         \text{ } & \text{ }    & \text{ } \\
           \ \begin{minipage}{4cm} For a specific fixed pair of dynamically defined Cantor sets $(C, K)$ \end{minipage} & \begin{minipage}{4cm} Explicit conditions are given in \cite{HS} \end{minipage}  & \begin{minipage}{4cm} There exists an example of two dynamically defined Cantor sets such that $C+K$ is a Cantor set of positive measure, \cite{Sa}. No verifiable criteria are known when thickness arguments do not help. \end{minipage}  \\
              \text{ } & \text{ }    & \text{ } \\
                        \hline

        \end{tabular}

}
\end{center}

\

One should also mention the results \cite{LM, Mas, PeSch, SS} in the spirit of Marstrand's theorem on properties of sum sets of the form $C_1 + \lambda C_2$ (which is equivalent to the projection of the product $C_1 \times C_2$ along the line of slope $\lambda$).

\medskip

In many applications a dynamically defined Cantor sets appears as the intersection of the stable lamination of some hyperbolic horseshoe with a transversal. More specifically, suppose that $f : M^2 \to M^2$ is a $C^r$-diffeomorphism, $r \ge 2$, and $\Lambda \subset M^2$ is a hyperbolic horseshoe (i.e., a totally disconnected locally maximal invariant compact set such that there exists an invariant splitting $T_\Lambda = E^s \oplus E^u$ so that along the stable subbundle $\{E^s\}$, the differential $Df$ contracts uniformly, and along $\{E^u\}$, the differential of the inverse $Df^{-1}$ contracts uniformly). Then
$$
W^s(\Lambda) = \{ x \in M^2 : \mathrm{dist} \, (f^n(x), \Lambda) \to 0\ \text{as}\ n \to +\infty\}
$$
consists of stable manifolds $W^s(\Lambda) = \bigcup_{x \in \Lambda} W^s(x)$ and locally looks like a product of a Cantor set with an interval. If $f = f_{\lambda^*} \in \{f_\lambda\}_{\lambda \in J = (\lambda_0, \lambda_1)}$ is an element of a smooth family of diffeomorphisms, then there exists a family of horseshoes $\{ \Lambda_\lambda \}, f_\lambda(\Lambda_\lambda) = \Lambda_\lambda$, for parameters $\lambda$ sufficiently close to the initial $\lambda^* \in J$. Suppose that $L \subset M^2$ is a line transversal to every leaf in $W^s(\Lambda_\lambda)$, $\lambda \in J$, with compact intersection $L \cap W^s(\Lambda_\lambda)$. The intersection $C_\lambda = L \cap W^s(\Lambda_\lambda)$ is a $\lambda$-dependent dynamically defined Cantor set. The lamination $\{ W^s(x) \}$ consists of $C^r$ leaves, but in general one cannot include it in a foliation of smoothness better than $C^{1+\alpha}$ (even for $C^\infty$ or real analytic $f$). That justifies the traditional assumption on $C^{1+\alpha}$ smoothness of generators of a dynamically defined Cantor set\footnote{Notice that $C^1$-smoothness is usually too weak since it does not allow one to use distortion property arguments; see  \cite{Moreira11, U95}  for some results on sums of $C^1$ Cantor sets.}. This prevents us from using Theorem \ref{t.sums} in the context above. Nevertheless, the analog of Theorem~\ref{t.sums} holds for families of Cantor sets $\{C_{\lambda}\}$ obtained via the described construction:

\begin{theorem}\label{t.sumhyp}
Suppose that $\{ f_\lambda \}_{\lambda \in J = (\lambda_0, \lambda_1)}$, $f_\lambda : M^2 \to M^2$, is a $C^2$-family of $C^2$-diffeomorphisms with uniformly {\rm (}in $\lambda${\rm )} bounded $C^2$ norms. Let $\{ \Lambda_\lambda \}_{\lambda \in J}$ be a family of hyperbolic horseshoes, and  $\{ L_\lambda \}_{\lambda \in J}$ be a smooth family of curves parameterized by $\gamma_\lambda : \mathbb{R} \to M^2$, transversal to $W^s(\Lambda_\lambda)$, with compact $C_\lambda = \gamma_\lambda^{-1}(L_\lambda \cap W^s(\Lambda_\lambda))$. Assume that
\begin{equation}\label{e.cond1}
\frac{d}{d\lambda} \dim_\mathrm{H} C_\lambda \ne 0 \ \text{for all}\ \ \lambda \in J.
\end{equation}
If $K \subset \mathbb{R}$ is a compact set such that
\begin{equation}\label{e.cond2}
\dim_\mathrm{H} C_\lambda + \dim_\mathrm{H} K > 1 \ \ \ \text{for all}\ \ \ \lambda\in J,
\end{equation}
then $\mathrm{Leb} (C_\lambda + K) > 0$ for Lebesgue almost every $\lambda\in J$.
\end{theorem}

Notice that Theorem~\ref{t.sumhyp} implies Theorem~\ref{t.sums}. Indeed, under the assumptions of Theorem~\ref{t.sums} one can construct a family of horseshoes and curves as in Theorem~\ref{t.sumhyp} that produce the same family of Cantor sets $\{C_\lambda\}$. We prove Theorem~\ref{t.sumhyp} in Section~\ref{s.sumproof}.

\subsection{An Application to the Square Fibonacci Hamiltonian}

The square Fibonacci Hamiltonian is the bounded self-adjoint operator
\begin{multline}\nonumber
[H^{(2)}_{\lambda_1, \lambda_2, \omega_1, \omega_2} \psi] (m,n) =  \psi(m+1,n) + \psi(m-1,n) + \psi(m,n+1) + \psi(m,n-1) + \\  + \Big(\lambda_1  \chi_{[1-\alpha , 1)}(m\alpha + \omega_1 \!\!\!\!\!\! \mod 1) + \lambda_2 \chi_{[1-\alpha , 1)}(n\alpha + \omega_2 \!\!\!\!\!\! \mod 1) \Big) \psi(m,n)
\end{multline}
in $\ell^2(\Z^2)$, with $\alpha=\frac{\sqrt{5}-1}{2}$, coupling constants $\lambda_1, \lambda_2 > 0$ and phases $\omega_1, \omega_2 \in \T = \R/\Z$. The standard Fibonacci Hamiltonian is the bounded self-adjoint operator
$$
[H^{(1)}_{\lambda, \omega} \psi] (n) =  \psi(n+1) + \psi(n-1) + \lambda_1  \chi_{[1-\alpha , 1)}(n\alpha + \omega_1 \!\!\! \mod 1) \psi(n)
$$
in $\ell^2(\Z)$, again with the coupling constant $\lambda > 0$ and the phase $\omega \in \T$. For a recent survey of the spectral theory of the Fibonacci Hamilonian and the square Fibonacci Hamiltonian, see \cite{DEG}.

Using the minimality of an irrational rotation and strong operator convergence, one can readily see that the spectra of these operators are phase-independent. That is, there are compact subsets $\Sigma_\lambda$ and $\Sigma_{\lambda_1, \lambda_2}$ of $\R$ such that
\begin{align*}
\sigma(H^{(1)}_{\lambda, \omega}) & = \Sigma_\lambda \quad \; \; \quad \text{ for every } \omega \in \T, \\
\sigma(H^{(2)}_{\lambda_1, \lambda_2, \omega_1, \omega_2}) & = \Sigma_{\lambda_1, \lambda_2} \quad \text{ for every } \omega_1, \omega_2 \in \T.
\end{align*}

The density of states measures associated with these operator families are defined as follows,
$$
\int_\R g(E) \, d\nu_{\lambda_1, \lambda_2}(E) = \int_{\T} \int_\T \langle \delta_0 , g(H^{(2)}_{\lambda_1, \lambda_2, \omega_1, \omega_2}) \delta_0 \rangle_{\ell^2(\Z^2)} \, d\omega_1 \, d\omega_2.
$$
and
$$
\int_\R g(E) \, d\nu_\lambda(E) = \int_{\T} \langle \delta_0 , g(H^{(1)}_{\lambda, \omega}) \delta_0 \rangle_{\ell^2(\Z)} \, d\omega.
$$
It is a standard result from the theory of ergodic Schr\"odinger operators that $\Sigma_\lambda = \mathrm{supp} \, \nu_\lambda$ and $\Sigma_{\lambda_1, \lambda_2} = \mathrm{supp} \, \nu_{\lambda_1, \lambda_2}$, where $\mathrm{supp} \, \nu$ denotes the topological support of the measure $\nu$.

The theory of separable operators (see, e.g., \cite[Appendix]{DGS} and \cite[Sections~II.4 and VIII.10]{RS}) quickly implies that
\begin{equation}\label{e.sumofspectra}
\Sigma_{\lambda_1, \lambda_2} = \Sigma_{\lambda_1} + \Sigma_{\lambda_2}
\end{equation}
and
\begin{equation}\label{e.convolutionofdos}
\nu_{\lambda_1, \lambda_2} = \nu_{\lambda_1} \ast \nu_{\lambda_2},
\end{equation}
where the convolution of measures is defined by
$$
\int_\R g(E) \, d(\mu \ast \nu)(E) = \int_\R \int_\R g(E_1 + E_2) \, d\mu(E_1) \, d\nu(E_2).
$$

It was shown in \cite{DGY} that for every $\lambda > 0$, the set $\Sigma_\lambda$ is a dynamically defined Cantor set (see \cite{Can, Cas, DG09} for earlier partial results for sufficiently small or large values of $\lambda$). In particular, its box counting dimension exists, coincides with its Hausdorff dimension, and the common value belongs to $(0,1)$. As was pointed out above, a particular consequence of this is that if $(\lambda_1,\lambda_2) \in \R_+^2$ is such that $\dim_\mathrm{H} \Sigma_{\lambda_1} + \dim_\mathrm{H} \Sigma_{\lambda_2} < 1$, then $\Sigma_{\lambda_1, \lambda_2}$ has zero Lebesgue measure. Here we are able to prove the following companion result:

\begin{theorem}\label{t.positivemeasureSPEC}
Suppose that for all pairs $(\lambda_1,\lambda_2)$ in some open set $U \subset \R_+^2$, we have $\dim_\mathrm{H} \Sigma_{\lambda_1} + \dim_\mathrm{H} \Sigma_{\lambda_2} > 1$. Then, for Lebesgue almost all pairs $(\lambda_1,\lambda_2) \in U$, $\Sigma_{\lambda_1, \lambda_2}$ has positive Lebesgue measure.
\end{theorem}

Combining results from \cite{DGS} and \cite{DGY}, it follows that for Lebesgue almost all pairs $(\lambda_1,\lambda_2) \in \R_+^2$ in the region where $\dim_\mathrm{H} \nu_{\lambda_1} + \dim_\mathrm{H} \nu_{\lambda_2} > 1$, the measure $\nu_{\lambda_1, \lambda_2}$ is absolutely continuous. We are also able here to prove a companion result for the latter statement:

\begin{theorem}\label{t.singularIDS}
Suppose that $\dim_\mathrm{H} \nu_{\lambda_1} + \dim_\mathrm{H} \nu_{\lambda_2} < 1$. Then, $\nu_{\lambda_1, \lambda_2}$ is singular, that is, it is supported by a set of zero Lebesgue measure.
\end{theorem}

In addition, it was shown in \cite{DGY} that for every $\lambda > 0$, we have
\begin{equation}\label{e.strictinequality}
0 < \dim_\mathrm{H} \nu_{\lambda} < \dim_\mathrm{H} \Sigma_{\lambda} < 1.
\end{equation}
This shows in particular that the curves
$$
\{ (\lambda_1,\lambda_2) \in \R_+^2 : \dim_\mathrm{H} \nu_{\lambda_1} + \dim_\mathrm{H} \nu_{\lambda_2} = 1 \}
$$
and
$$
\{ (\lambda_1,\lambda_2) \in \R_+^2 : \dim_\mathrm{H} \Sigma_{\lambda_1} + \dim_\mathrm{H} \Sigma_{\lambda_2} = 1 \}
$$
are disjoint. The complement of the union of these curves consists of three regions, in which we have three different kinds of spectral behavior due to the results above and the discussion preceding each of them. We summarize these findings and make the global picture explicit in the following corollary.

\begin{coro}
Consider the following three regions in $\R_+^2$:
\begin{align*}
U_{\mathrm{acds}} & = \{ (\lambda_1,\lambda_2) \in \R_+^2 : \dim_\mathrm{H} \nu_{\lambda_1} + \dim_\mathrm{H} \nu_{\lambda_2} > 1 \}, \\
U_{\mathrm{pmsd}} & = \{ (\lambda_1,\lambda_2) \in \R_+^2 : \dim_\mathrm{H} \Sigma_{\lambda_1} + \dim_\mathrm{H} \Sigma_{\lambda_2} > 1 \text{ and } \dim_\mathrm{H} \nu_{\lambda_1} + \dim_\mathrm{H} \nu_{\lambda_2} < 1 \}, \\
U_{\mathrm{zmsp}} & = \{ (\lambda_1,\lambda_2) \in \R_+^2 : \dim_\mathrm{H} \Sigma_{\lambda_1} + \dim_\mathrm{H} \Sigma_{\lambda_2} < 1  \}.
\end{align*}
Then, the following statements hold:
\begin{itemize}

\item[{\rm (a)}] Each of the regions $U_{\mathrm{acds}}$, $U_{\mathrm{pmsd}}$, $U_{\mathrm{zmsp}}$ is open and non-empty.

\item[{\rm (b)}] The regions $U_{\mathrm{acds}}$, $U_{\mathrm{pmsd}}$, $U_{\mathrm{zmsp}}$ are disjoint and the union of their closures covers the parameter space $\R_+^2$.

\item[{\rm (c)}] For Lebesgue almost every $(\lambda_1,\lambda_2) \in U_{\mathrm{acds}}$, $\nu_{\lambda_1, \lambda_2}$ is absolutely continuous, and hence $\Sigma_{\lambda_1, \lambda_2}$ has positive Lebesgue measure.

\item[{\rm (d)}] For every $(\lambda_1,\lambda_2) \in U_{\mathrm{pmsd}}$, $\nu_{\lambda_1, \lambda_2}$ is singular, but for Lebesgue almost every  $(\lambda_1,\lambda_2) \in U_{\mathrm{pmsd}}$, $\Sigma_{\lambda_1, \lambda_2}$ has positive Lebesgue measure.

\item[{\rm (e)}] For every $(\lambda_1,\lambda_2) \in U_{\mathrm{zmsp}}$, $\Sigma_{\lambda_1, \lambda_2}$ has zero Lebesgue measure, and hence $\nu_{\lambda_1, \lambda_2}$ is singular.

\end{itemize}
\end{coro}

\begin{remark}
(a) The potential of the Fibonacci Hamiltonian may be generated by the Fibonacci substitution $a \mapsto ab$, $b \mapsto a$. This substitution is the most prominent example of an invertible two-letter substitution. We believe that, using \cite{G14, M}, the results above may be generalized to the case where the Fibonacci substitution is replaced by a general invertible two-letter substitution. \\
(b) We expect that similar phenomena can appear also in other models, such as for example the labyrinth model, or the square off-diagonal (or tridiagonal) Fibonacci Hamiltonian.
\end{remark}

The coexistence of positive measure spectrum and singular density of states measure is a rather unusual phenomenon. Until very recently it was an open problem whether this can even occur in the context of Schr\"odinger operators. The existence of Schr\"odinger operators with quasi-periodic potentials exhibiting this phenomenon was shown in \cite{ADZ}. However, the examples given in that paper are somewhat artificial, and ``typical'' quasi-periodic Schr\"odinger operators are not expected to have these two properties. The examples provided by the Square Fibonacci Hamiltonian with parameters in $U_{\mathrm{pmsd}}$, on the other hand, are not artificial at all, but rather correspond to operators that are arguably physically relevant. Moreover the phenomenon is made possible by and is closely connected to the strict inequality between $\dim_\mathrm{H} \nu_{\lambda}$ and $\dim_\mathrm{H} \Sigma_{\lambda}$, as stated in \eqref{e.strictinequality}, which was originally conjectured by Barry Simon and finally proved in \cite{DGY} (see \cite{DG12} for an earlier partial result for sufficiently small values of $\lambda$).

\section{Sums of Dynamically Defined Cantor Sets}\label{s.sumproof}

In this section we prove Theorem~\ref{t.sumhyp}. The proof is based on Theorem~3.7 from \cite{DGS}. The setting there is the following.

Suppose $J \subset \R$ is a compact interval, and  $f_{\lambda} : M^2 \to M^2$, $\lambda \in J$, is a smooth family of smooth surface diffeomorphisms. Specifically, we require $f_\lambda(p)$ to be $C^2$-smooth with respect to both $\lambda$ and $p$, with a finite $C^2$-norm. Also, we assume that $f_{\lambda} : M^2 \to M^2$, $\lambda \in J$, has a locally maximal transitive totally disconnected hyperbolic set $\Lambda_\lambda$ that depends continuously on the parameter.

Let $\gamma_\lambda : \R \to M^2$ be a family of smooth curves, smoothly depending on the parameter, and $L_\lambda = \gamma_\lambda(\R)$. Suppose that the stable manifolds of $\Lambda_\lambda$ are transversal to $L_\lambda$.

\begin{lemma}[Lemma 3.1 from \cite{DGS}]\label{l.pi}
There is a Markov partition of $\Lambda_\lambda$ and a continuous family of projections $\pi_\lambda:\Lambda_\lambda\to L_\lambda$ along stable manifolds of $\Lambda_\lambda$ such that for any two distinct elements of the Markov partition, their images under $\pi_\lambda$ are disjoint.
\end{lemma}

Suppose $\sigma_A : \Sigma^\ell_A \to \Sigma^\ell_A$ is a topological Markov chain, which for every $\lambda \in J$ is conjugated to $f_{\lambda} : \Lambda_{\lambda} \to \Lambda_{\lambda}$ via the conjugacy $H_{\lambda} : \Sigma^\ell_A \to \Lambda_{\lambda}$. Let $\mu$ be an ergodic probability measure for $\sigma_A : \Sigma_A^\ell \to \Sigma_A^\ell$ such that $h_{\mu}(\sigma_A)>0$. Set $\mu_\lambda = H_\lambda (\mu)$, then $\mu_\lambda$ is an ergodic invariant measure for $f_\lambda : \Lambda_\lambda \to \Lambda_\lambda$.

Let $\pi_\lambda : \Lambda_\lambda \to L_\lambda$ be the continuous family of continuous projections along the stable manifolds of $\Lambda_\lambda$ provided by Lemma \ref{l.pi}. Set $\nu_\lambda = \gamma^{-1}_\lambda \circ \pi_\lambda (\mu_\lambda) = \gamma^{-1}_\lambda \circ \pi_\lambda \circ H_\lambda (\mu)$.

In this setting the following theorem holds.

\begin{theorem}[Theorem 3.7 from \cite{DGS}]\label{t.ac2}
Suppose that $J$ is a compact interval so that $\left| \frac{d}{d\lambda} \mathrm{Lyap}^{u}(\mu_\lambda) \right| \ge \delta > 0$ for some $\delta > 0$ and all $\lambda \in J$. Then for any compactly supported exact-dimensional measure $\eta$ on $\R$ with $$\dim_\mathrm{H} \eta + \dim_\mathrm{H} \nu_\lambda > 1$$  for all $\lambda \in J$, the convolution $\eta * \nu_\lambda$ is absolutely continuous with respect to Lebesgue measure for Lebesgue almost every $\lambda \in J$.
\end{theorem}

\begin{remark}
In fact, in Theorem \ref{t.ac2} the condition on exact dimensionality of the measure $\eta$ can be replaced by the following one (and this is the only consequence of exact dimensionality of $\eta$ that was used in the proof of Theorem \ref{t.ac2} in \cite{DGS}):
\begin{itemize}
\item There are $C > 0$ and $d > 0$ such that for every $x \in \mathbb{R}$ and $r > 0$, we have $\eta(B_r(x)) \leq C r^{d}$.
\end{itemize}
\end{remark}

\begin{proof}[Proof of Theorem~\ref{t.sumhyp}.]
The condition $\dim_\mathrm{H} C_{\lambda} + \dim_\mathrm{H} K > 1$ trivially implies that $\dim_\mathrm{H} K > 0$. By Frostman's Lemma (see, e.g., \cite[Theorem~8.8]{Mattila}), for every $d < \dim_\mathrm{H} K$, there exist a Borel measure $\eta$ on $\R$ with $\eta(K) = 1$ and a constant $C$ such that
\begin{equation}\label{FrostmanMainC}
\eta(B_r(x)) \leq C r^{d} \quad \text{for every } x \in \R \text{ and } r > 0.
\end{equation}

We will show that for every $\lambda_0 \in J$, there exists $\varepsilon = \varepsilon(\lambda_0) > 0$ such that $\mathrm{Leb} (C_\lambda + K) > 0$ for Lebesgue almost every $\lambda \in (\lambda_0 - \varepsilon, \lambda_0 + \varepsilon) \cap J$. This will imply Theorem~\ref{t.sumhyp}.

Fix $\lambda_0 \in J$. Let $\mu_{\lambda_0}$ be the equilibrium measure on $\Lambda_{\lambda_0}$ that corresponds to the potential $-\dim_\mathrm{H} C_{\lambda_0} \log |Df_{\lambda_0}|_{E^u}|$. Then (see \cite{MM}), the measure $\mu_{\lambda_0}$ is a measure of maximal (unstable) dimension, that is, $\dim_\mathrm{H} \pi_{\lambda_0} (\mu_{\lambda_0}) = \dim_\mathrm{H} C_{\lambda_0}$. Denote by $\nu_{\lambda_0}$ the projection $\pi_{\lambda_0}(\mu_{\lambda_0})$. In order to mimic the setting of Theorem~\ref{t.ac2}, set $\mu = H^{-1}_{\lambda_0}(\mu_{\lambda_0})$. Then $\mu$ is an invariant probability measure for the shift $\sigma_A : \Sigma_A^l \to \Sigma_A^l$. Let us denote $\mu_\lambda = H_\lambda(\mu)$ and $\nu_\lambda = \pi_\lambda(\mu_\lambda)$. There exists a canonical family of conjugacies $\mathcal{H}_{\lambda_1, \lambda_2} : \Lambda_{\lambda_1} \to \Lambda_{\lambda_2}$, $(\lambda_1, \lambda_2) \in J \times J$, so that $\mathcal{H}_{\lambda_1, \lambda_2} \circ f_{\lambda_1} = f_{\lambda_2} \circ \mathcal{H}_{\lambda_1, \lambda_2}$. It is well known (see, for example, Theorem 19.1.2 from \cite{KaH}) that each of the maps $\mathcal{H}_{\lambda_1, \lambda_2}$ is H\"older continuous. Moreover, the H\"older exponent tends to one as $|\lambda_1 - \lambda_2| \to 0$; see \cite{PV}.  As a result, we conclude that for any $\lambda$ sufficiently close to $\lambda_0$, we have
$$
\dim_\mathrm{H} \nu_\lambda + d > 1
$$
for a suitable $d$ that is chosen sufficiently close to $\dim_\mathrm{H} K$ and for which we have \eqref{FrostmanMainC} with suitable $\eta$ and $C$.

In order to apply Theorem~\ref{t.ac2} we need to show that  $\left| \frac{d}{d\lambda} \mathrm{Lyap}^u(\mu_\lambda) \right| \ge \delta > 0$. But due to \cite{Man} we know that
$$
\mathrm{Lyap}^u \mu_\lambda = \frac{h_{\mu_\lambda}}{\dim_\mathrm{H} \nu_{\lambda}},
$$
where $h_{\mu_{\lambda_0}} = h_{\mu}$ is the entropy of the invariant measure $\mu_\lambda$ (which is by construction independent of $\lambda$). Notice also that $\mathrm{Lyap}^u \mu_\lambda$ is a $C^1$ smooth function of $\lambda$. Indeed, the center-stable and center-unstable manifolds of the partially hyperbolic invariant set of the map $(\lambda, p) \mapsto (\lambda, f_\lambda(p))$ are $C^2$-smooth, hence
$$
\mathrm{Lyap}^u \mu_\lambda = \int_{\Lambda_\lambda} \log |Df_\lambda|_{E^u}| \, d\mu_\lambda = \int_{\Sigma_A^l} \log|Df_\lambda(H_\lambda(\omega))|_{E^u}| \, d\mu(\omega)
$$
is a $C^1$-smooth function of $\lambda \in J$. 

Finally, consider $\dim_\mathrm{H} C_\lambda$ and $\dim_\mathrm{H} \nu_\lambda$ as functions of $\lambda$; see Fig. \ref{fig.2}. Due to \cite{Ma} we know that $\dim_\mathrm{H} C_\lambda$ is a $C^1$-function of $\lambda$. Without loss of generality we can assume that $\frac{d}{d\lambda}\dim_\mathrm{H} C_\lambda \ge \delta > 0$ for some $\delta > 0$. Since $\mathrm{supp} \, \nu_\lambda \subseteq C_\lambda$, we have $\dim_\mathrm{H} \nu_\lambda \le \dim_\mathrm{H} C_\lambda$. By construction we have $\dim_\mathrm{H} \nu_{\lambda_0} = \dim_\mathrm{H} C_{\lambda_0}$. This implies that $\frac{d}{d\lambda}|_{\lambda = \lambda_0} \dim_\mathrm{H} \nu_{\lambda} = \frac{d}{d\lambda}|_{\lambda = \lambda_0} \dim_\mathrm{H} C_{\lambda} \ge \delta > 0$, and hence for some $\varepsilon > 0$, $\frac{d}{d\lambda} \dim_\mathrm{H} \nu_{\lambda} \ge \frac{\delta}{2} > 0$ for $\lambda \in (\lambda_0 - \varepsilon, \lambda_0 + \varepsilon)$. Now we can apply Theorem~\ref{t.ac2} to the measures $\eta$ and $\nu_\lambda$, and get that for Lebesgue almost every $\lambda \in (\lambda_0 - \varepsilon, \lambda_0 + \varepsilon)$, the convolution $\eta * \nu_\lambda$ is absolutely continuous with respect to Lebesgue measure, and hence $\mathrm{Leb}(C_\lambda + K) > 0$.
\end{proof}
\begin{figure}\label{fig.2}
\begin{center}
   \includegraphics[scale=.7]{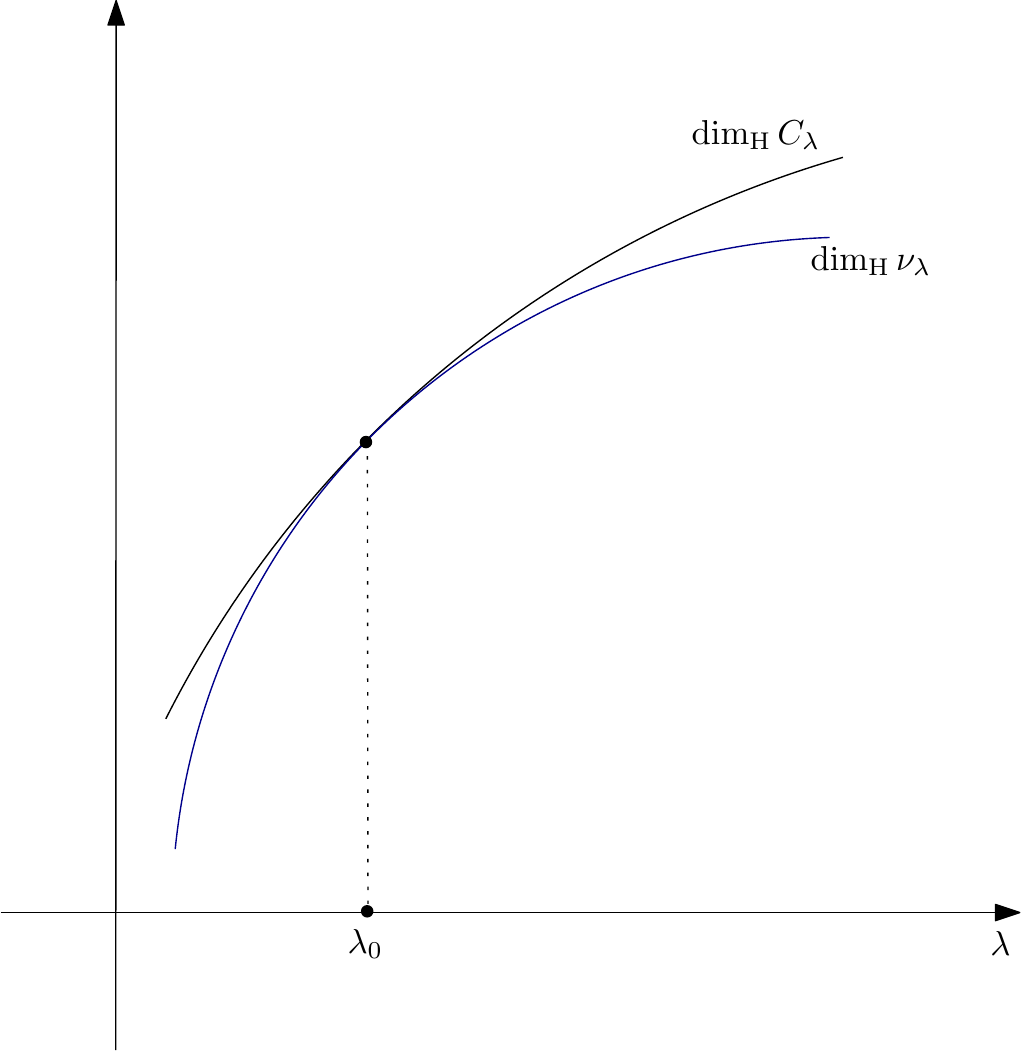}\caption{Graphs of $\dim_\mathrm{H} C_\lambda$ and $\dim_{H} \nu_\lambda$ as functions of $\lambda$}
   \end{center}
\end{figure}

\section{The Square Fibonacci Hamiltonian}

\subsection{A Dynamical Description of the Spectrum of the Fibonacci Hamiltonian}

There is a fundamental connection between the spectral properties of the Fibonacci Hamiltonian and the dynamics of the \textit{trace map}
\begin{equation}\label{e.tracemapdef}
T : \Bbb{R}^3 \to \Bbb{R}^3, \; T(x,y,z)=(2xy-z,x,y).
\end{equation}
The function $G(x,y,z) = x^2+y^2+z^2-2xyz-1$ is invariant\footnote{$G$ is usually called the Fricke-Vogt invariant.} under the action of $T$, and hence $T$ preserves the family of cubic surfaces\footnote{The surface $S_0$ is known as Cayley cubic.}
\begin{equation}\label{e.surfacelambdadef}
S_\lambda = \left\{(x,y,z)\in \Bbb{R}^3 : x^2+y^2+z^2-2xyz=1+ \frac{\lambda^2}{4} \right\}.
\end{equation}
It is therefore natural to consider the restriction $T_{\lambda}$ of the trace map $T$ to the invariant surface $S_\lambda$. That is, $T_{\lambda}:S_\lambda \to S_\lambda$, $T_{\lambda}=T|_{S_\lambda}$. We denote by $\Lambda_{\lambda}$ the set of points in $S_\lambda$ whose full orbits under $T_{\lambda}$ are bounded (it follows from \cite{Can, R} that $\Lambda_\lambda$ is equal to the non-wandering set of $T_\lambda$; compare the discussion in \cite{DG12}).

Denote by $\ell_\lambda$ the line
\begin{equation}\label{e.loic}
\ell_\lambda = \left\{ \left(\frac{E-\lambda}{2}, \frac{E}{2}, 1 \right) : E \in \R \right\}.
\end{equation}
It is easy to check that $\ell_\lambda \subset S_\lambda$. The key to the fundamental connection between the spectral properties of the Fibonacci Hamiltonian and the dynamics of the trace map is the following result of S\"ut\H{o} \cite{S87}. An energy $E \in \R$ belongs to the spectrum $\Sigma_\lambda$ of the Fibonacci Hamiltonian if and only if the positive semiorbit of the point $(\frac{E-\lambda}{2}, \frac{E}{2}, 1)$ under iterates of the trace map $T$ is bounded.

It turns out that for every $\lambda > 0$, $\Lambda_\lambda$ is a locally maximal compact transitive hyperbolic set of $T_{\lambda} : S_\lambda \to S_\lambda$; see \cite{Can, Cas, DG09}. Moreover, it was shown in \cite{DGY} that for every $\lambda > 0$, the line of initial conditions $\ell_\lambda$ intersects $W^s(\Lambda_\lambda)$ transversally. Thus, we are essentially in the setting in which Theorem~\ref{t.sumhyp} applies. The only minor difference is that in the present setting, the surface $S_\lambda$ depends formally on $\lambda$, while it is $\lambda$-independent in the setting of Theorem~\ref{t.sumhyp}. After partitioning the parameter space into smaller intervals if necessary, we can then consider a small $\lambda$-interval, choose a $\lambda_0$ in it, and then conjugate with smooth projections of $S_\lambda$ to $S_{\lambda_0}$.

\subsection{Proof of Theorem~\ref{t.positivemeasureSPEC}}

Using the connection between the spectrum of the (one-dimensional) Fibonacci Hamiltonian and the dynamics of the trace map, we can now derive Theorem~\ref{t.positivemeasureSPEC} from Theorem~\ref{t.sumhyp}.

\begin{proof}[Proof of Theorem~\ref{t.positivemeasureSPEC}.]
It clearly suffices to work locally in $U$. That is, we consider a rectangular box $B = \{ (\lambda_1,\lambda_2) : a < \lambda_1 < b, \, c < \lambda_2 < d \}$ inside $U$ and prove that for Lebesgue almost every $(\lambda_1,\lambda_2) \in B$, $\Sigma_{\lambda_1,\lambda_2}$ has positive Lebesgue measure. To accomplish this, it suffices to show that for every fixed $\lambda_2 \in (c,d)$, $\Sigma_{\lambda_1,\lambda_2}$ has positive Lebesgue measure for Lebesgue almost every $\lambda_1 \in (a,b)$.

The set $\Sigma_{\lambda_2}$ will play the role of the set $K$ in Theorem~\ref{t.sumhyp}. By the analyticity of $\lambda_1 \mapsto \dim_\mathrm{H} \Sigma_{\lambda_1}$, we can subdivide $(a,b)$ into intervals, on the interiors of which we have the condition
$$
\frac{d}{d\lambda_1} \dim_\mathrm{H} \Sigma_{\lambda_1} \ne 0.
$$
This ensures that condition \eqref{e.cond1} in Theorem~\ref{t.sumhyp} holds. Condition \eqref{e.cond2} in Theorem~\ref{t.sumhyp} holds since we work inside $U$. All the other assumptions in Theorem~\ref{t.sumhyp} hold by the discussion in the previous subsection. Thus we may apply Theorem~\ref{t.sumhyp} and obtain the desired statement.
\end{proof}

\subsection{Proof of Theorem~\ref{t.singularIDS}}

Let us begin by recalling some basic concepts from measure theory and fractal geometry; the standard texts \cite{Falc,Mattila} can be consulted for background information. Suppose $\mu$ is a finite Borel measure on $\R^d$. The lower Hausdorff dimension, resp.\ the upper Hausdorff dimension, of $\mu$ are given by
\begin{align}
\label{e.lowerdimmeas}  \dim_\mathrm{H}^-(\mu) & = \inf \{ \dim_\mathrm{H}(S) : \mu(S) > 0 \}, \\
\label{e.upperdimmeas} \dim_\mathrm{H}^+(\mu) & = \inf \{ \dim_\mathrm{H}(S) : \mu(\R^d \setminus S) = 0 \}.
\end{align}
These dimensions can be interpreted in the following way. The measure $\mu$ gives zero weight to every set $S$ with $\dim_\mathrm{H}(S) <
\dim_\mathrm{H}^-(\mu)$ and, for every $\varepsilon > 0$, there is a set $S$ with $\dim_\mathrm{H}(S) < \dim_\mathrm{H}^+(\mu) + \varepsilon$ that supports
$\mu$ (i.e., $\mu(\R \setminus S) = 0$).

For $x \in \R^d$ and $\varepsilon > 0$, we denote the open ball with radius $\varepsilon$ and center $x$ by $B(x,\varepsilon)$. The lower scaling exponent of $\mu$ at $x$ is given by
$$
\alpha_\mu^-(x) = \liminf_{\varepsilon \to 0} \frac{\log \mu (B(x,\varepsilon))}{\log \varepsilon}.
$$
For $\mu$-almost every $x$, $\alpha^-(x) \in [0,d]$. Moreover, we have
\begin{align}
\label{e.dimlechar5} \dim_\mathrm{H}^- (\mu) & = \mu - \mathrm{essinf} \, \alpha_\mu^- \, \equiv \sup \{ \alpha : \alpha_\mu^-(x) \ge \alpha \text{ for $\mu$-almost every } x \}, \\
\label{e.dimlechar6} \dim_\mathrm{H}^+ (\mu) & = \mu - \mathrm{esssup} \, \alpha_\mu^- \equiv \inf \{ \alpha : \alpha_\mu^-(x) \le \alpha \text{ for $\mu$-almost every } x \},
\end{align}
compare \cite[Propositions~10.2 and 10.3]{Falc2}.

One can also consider the upper scaling exponent of $\mu$ at $x$,
$$
\alpha_\mu^+(x) = \limsup_{\varepsilon \to 0} \frac{\log \mu (B(x,\varepsilon))}{\log \varepsilon},
$$
which also belongs to $[0,d]$ for $\mu$-almost every $x$. The measure $\mu$ is called exact-dimensional if there is a number $\dim \mu \in [0,d]$ such that $\alpha_\mu^+(x) = \alpha_\mu^-(x) = \dim \mu$ for $\mu$-almost every $x \in \R^d$. In this case, it of course follows that $\dim_\mathrm{H}^+(\mu) = \dim_\mathrm{H}^-(\mu) = \dim \mu$, and tangentially we note that the common value also coincides with the upper and lower packing dimension of $\mu$, which are defined analogously by replacing the Hausdorff dimension of a set in the above definitions by the packing dimension; see \cite{Falc, Falc2, Mattila} for further details.

\bigskip

We are now ready to prove Theorem~\ref{t.singularIDS}. In fact, the theorem will follow quickly from known results once we have established the following simple lemma.

\begin{lemma}\label{l.singularconvolution}
Suppose $\nu_1$ and $\nu_2$ are compactly supported exact-dimensional measures on $\R$ of dimension $d_1$ and $d_2$, respectively. If $d_1 + d_2 < 1$, then the convolution $\nu_1 \ast \nu_2$ is singular.
\end{lemma}

\begin{proof}
Note first that the product measure $\nu_1 \times \nu_2$ is exact-dimensional with dimension $d_1 + d_2$. Moreover, the convolution $\nu_1 \ast \nu_2$ can be obtained from $\nu_1 \times \nu_2$ by projection, that is,
$$
\nu_1 \ast \nu_2 (B) = \nu_1 \times \nu_2 \{ (x,y) \in \R^2 : x+y \in B \}.
$$
It follows that for $\nu_1 \ast \nu_2$-almost every $x \in \R$, the lower scaling exponent
$$
\alpha_{\nu_1 \ast \nu_2}^-(x) = \liminf_{\varepsilon \downarrow 0} \frac{\log \left( \nu_1 \ast \nu_2 \left( (x-\varepsilon,x+\varepsilon) \right) \right)}{\log \varepsilon}
$$
is bounded from above by $d_1 + d_2$. This implies that the upper Hausdorff dimension of $\nu_1 \ast \nu_2$,
\begin{align*}
\dim_\mathrm{H}^+(\nu_1 \ast \nu_2) & = \inf \{ \dim_\mathrm{H}(S) : \nu_1 \ast \nu_2(\R \setminus S) = 0 \} \\
& = \nu_1 \ast \nu_2 - \mathrm{esssup} \, \alpha_{\nu_1 \ast \nu_2}^- \\
& \equiv \inf \{ d : \alpha_{\nu_1 \ast \nu_2}^-(x) \le d \text{ for } \nu_1 \ast \nu_2 \text{-almost every } x \},
\end{align*}
is bounded from above by $d_1 + d_2$ (here we used \eqref{e.upperdimmeas} and \eqref{e.dimlechar6}). Since $d_1 + d_2 < 1$ by assumption, $\nu_1 \ast \nu_2$ has a support of Hausdorff dimension strictly less than one and hence of Lebesgue measure zero. This shows that $\nu_1 \ast \nu_2$ is singular.
\end{proof}

\begin{proof}[Proof of Theorem~\ref{t.singularIDS}.]
It was shown in \cite{DGY} that for every $\lambda > 0$, the density of states measure $\nu_\lambda$ is exact-dimensional. Thus, Theorem~\ref{t.singularIDS} is an immediate consequence of Lemma~\ref{l.singularconvolution}.
\end{proof}

\end{document}